\documentclass[12pt]{amsart}
\usepackage{main}
\usepackage[margin=0.5cm]{caption}

\def\Grad{\mathop{\rm grad}}
\def\Div{\mathop{\rm div}}
\def\bar{\overline}

\def\tilde{\widetilde}

\def\inter{\text{int}}

\title[Inverse obstacle problem and the BC method]{Solving an inverse obstacle problem for the wave equation by using the boundary control method}
\author{Lauri Oksanen}
\address{University of Helsinki, P.O. Box 68 FI-00014}
\email{lauri.oksanen@helsinki.fi}
\date{\today}
\begin{document}
\begin{abstract}
We introduced in \cite{Oksanen2011a} a method to locate discontinuities of a wave speed in dimension two from acoustic boundary measuments modelled by the hyperbolic Neumann-to-Dirichlet operator. Here we extend the method for sound hard obstacles in arbitrary dimension.
We present numerical experiments 
with simulated noisy data suggesting 
that the method is robust against measurement noise.
\end{abstract}
\maketitle

\section{Introduction}

Nondestructive obstacle reconstruction through wave propagation motivates a number of mathematical problems with several applications such as medical and seismic imaging.
There is a large body of literature concerning obstacle detection 
using time harmonic waves, and we refer the reader to the review articles \cite{Colton2000, Potthast2006} and to the monograph \cite{Kirsch2008}.
Recently there has been also interest in 
reconstruction methods from acoustic measurements in the time domain
\cite{Burkard2009, Chen2010, Lines2005, Luke2006}.
In this paper we present a numerical method of the latter type. 
We allow the background to be anisotropic and non-homogeneous but confine ourselves to the case of non-stationary acoustic waves 
and the scattering from sound-hard obstacles.

Let $M$ be a compact smooth manifold with smooth boundary $\p M$
and let $g$ be a smooth Riemannian metric tensor on $M$.
Let $\Sigma \subset M^\inter$ be a compact set with nonempty interior and smooth boundary,
and let $\mu \in C^\infty(M)$ be strictly positive.
We consider the following wave equation on $M$,
\begin{align}
\label{eq_wave}
&\p_t^2 u(t,x) - \Delta_{g, \mu} u(t, x) = 0, & (t,x) \in (0,\infty) \times (M \setminus \Sigma),
\\\nonumber& \p_{\nu, \mu} u(t,x) = f(t,x), & (t,x) \in (0,\infty) \times \p M,
\\\nonumber& \p_{\nu, \mu} u(t,x) = 0, & (t,x) \in (0,\infty) \times \p \Sigma,
\\\nonumber& u|_{t=0}(x) = 0, \quad \p_t u|_{t=0}(x)=0,  & x \in M \setminus \Sigma,
\end{align}
where $\Delta_{g, \mu}$ is the weighted Laplace-Beltrami operator
and $\p_{\nu, \mu}$ is the normal derivative corresponding to $\Delta_{g, \mu}$.
That is,
if we let $(g^{jk}(x))_{j,k=1}^n$ and $|g(x)|$ denote the inverse and determinant of $g(x)$ in local coordinates,
then we have
\begin{align*}
\Delta_{g, \mu} u
&= \mu^{-1} \Div (\mu \Grad u)
\\
&= \sum_{j,k=1}^n \mu(x)^{-1}|g(x)|^{-\frac 12}\frac {\p}{\p x^j} 
\ll( \mu(x)|g(x)|^{\frac 12}g^{jk}(x)\frac {\p u}{\p x^k} \rr),
\\
\p_{\nu, \mu} u 
&= \mu (\Grad u, \nu)_{T M \times T^* M}
= \sum_{j,k=1}^n \mu(x) \nu_k(x) g^{jk}(x) \frac{\p u}{\p x^j},
\end{align*}
where $\nu = (\nu_1, \dots, \nu_n)$ is the exterior co-normal vector of $\p M$
normalized with respect to $g$, that is, $\sum_{j,k=1}^m g^{jk}\nu_j\nu_k=1$.

Let us denote the solution of (\ref{eq_wave}) by $u^f(t,x) = u(t,x)$. 
For $T > 0$ and an open $\Gamma \subset \p M$
we define the operator 
\begin{align*}
\Lambda_{T, \Gamma} : f \mapsto u^f|_{(0,T) \times \Gamma},
\quad 
f \in C_0^\infty((0,T) \times \Gamma).
\end{align*}
The Neumann-to-Dirichlet operator $\Lambda_{T, \Gamma}$ models 
boundary measurements with acoustic sources 
and receivers on $\Gamma$. 
Let us assume that the metric tensor $g$ and the weight function $\mu$ are known but $\Sigma$ is unknown.
We consider a method to locate $\Sigma$ from the measurements $\Lambda_{T, \Gamma}$. 

Let us point out that if $M \subset \R^n$,
$g = c(x)^{-2} dx^2$ and $\mu(x) = c(x)^{n-2}$
where $c \in C^\infty(M)$ is strictly positive, then
$\Delta_{g, \mu} = c(x)^2 \Delta$, 
where $\Delta$ is the Euclidean Laplacian. Thus the isotropic wave equation,
\begin{align*}
\p_t^2 u - c(x)^2 \Delta u = 0,
\end{align*}
is covered by the theory.
The more general equation (\ref{eq_wave})
allows for an anisotropic wave speed to be modelled.

\subsection{Statement of the results}
Notice that the operator $\Delta_{g, \mu}$ with the domain $H^2(M) \cap H^1_0(M)$ is self-adjoint on
the space $L^2(M; \mu dV_g)$, where $dV_g$ is the Riemannian volume measure of $(M, g)$,
that is, $\mu dV_g = \mu |g|^{1/2} dx$ in local coordinates.
We call $\mu dV_g$ the measure corresponding to $\Delta_{g, \mu}$
and denote it also by $V$.

We define for a function $\tau : \p M \to \R$ the {\em domain of influence} with and without the obstacle,
\begin{align*} 
M_\Sigma(\tau) &:= \{x \in M \setminus \Sigma;\ \text{there is $y \in \p M$ such that $d_\Sigma(x,y) \le \tau(y)$}\},
\\
M(\tau) &:= \{x \in M;\ \text{there is $y \in \p M$ such that $d(x,y) \le \tau(y)$}\},
\end{align*} 
where $d_\Sigma$ is the Riemannian distance function of $(M \setminus \Sigma, g)$
and $d$ is that of $(M, g)$.
As $(M,g)$ is known, we can compute the shape of the domain of influence $M(\tau)$
for any $\tau : \p M \to \R$.
Our main theorem is the following:
\begin{theorem}
\label{thm_main}
Let $T > 0$ and let $\Gamma \subset \p M$ be open.
For a function $\tau$ in 
\begin{equation*}
C_{T}(\Gamma) := \{ \tau : \p M \to \R;\ 
\tau|_{\bar \Gamma} \in C(\bar \Gamma),\ 0 \le \tau \le T,\ \tau|_{\p M \setminus \bar \Gamma} = 0 \},
\end{equation*}
the volume $V(M_\Sigma(\tau))$
can be computed from $\Lambda_{2T, \Gamma}$ by solving a sequence of 
linear equations on $L^2((0,T) \times \Gamma)$.
Moreover, 
\begin{equation}
\label{eq_detection}
M(\tau)^\inter \cap \Sigma^\inter \ne \emptyset
\quad \text{if and only if} \quad
V(M_\Sigma(\tau)) < V(M(\tau)).
\end{equation}
\end{theorem}

Theorem \ref{thm_main} allows us to probe the obstacle
with the known domains of influence $M(\tau)$, $\tau \in C_{T}(\Gamma)$.
We will illustrate this probing method in Section \ref{sec_computations}
via numerical experiments in the two dimensional case. 

In Section \ref{sec_proof_of_the_th} we give a proof of Theorem \ref{thm_main} that 
is based on ideas from the boundary control method.
By using the boundary control method,
a smooth wave speed can be fully reconstructed from the Neumann-to-Dirichlet operator.
This uniqueness result is by Belishev \cite{Belishev1987} in the isotropic case 
and by Belishev and Kurylev \cite{Belishev1992} in the anisotropic case. 
We refer to the monograph \cite{Katchalov2001} and to the review article \cite{Belishev2007} for further details on the boundary control method.
The boundary control method depends on Tataru's hyperbolic unique continuation result \cite{Tataru1995}, whence it is expected to have 
only logarithmic type stability.
Also our result depends on \cite{Tataru1995}, however, 
we overcome the ill-posedness of the reconstruction problem by 
regularizing it carefully. 
The regularization stategy is a modification of that in \cite{Bingham2008},
and the iterative time-reversal control method introduced there can be adapted to give an efficient implementation of our method.


\section{Proof of the main theorem}
\label{sec_proof_of_the_th}

We begin by showing that the volumes $V(M_\Sigma(\tau))$, $\tau \in C_{T}(\Gamma)$,
can be computed from $\Lambda_{2T, \Gamma}$ by solving a sequence of 
linear equations on $L^2((0,T) \times \p M)$.
Our proof relies on general results from regularization theory
and it can also be adapted to simplify the arguments in \cite{Oksanen2011}.
We define the operator
\def\Op1{\mathbbm 1}
\begin{align*}
&K := J \Lambda_{2T, \Gamma} \Theta_{2T} - R \Lambda_{T, \Gamma} R J \Theta_{2T},
\end{align*}
where $\Theta_{2T}$ is the extension by zero from $(0, T)$ to $(0, 2T)$,
$R$ is the time reversal on $(0, T)$, that is $Rf(t) := f(T - t)$, and
\begin{align*}
&Jf(t) := \frac{1}{2} \int_t^{2T - t} f(s) ds,
\quad f \in L^2(0, 2T),\ t \in (0, T).
\end{align*}
We recall $K$ is a compact operator on $L^2((0, T) \times \Gamma)$
since, see \cite{Tataru1998},
\begin{align*}
\Lambda_{T, \Gamma} : L^2((0, T) \times \Gamma) \to H^{2/3}((0,T) \times \Gamma).
\end{align*}

Let $f \in C_0^\infty((0,T) \times \Gamma)$
and let $\phi \in C^\infty(M \setminus \Sigma)$.
Moreover, let $t \in (0, \infty)$ and integrate by parts
\begin{align}
\label{the_integration_by_parts}
&\p_t^2 (u^f(t), \phi)_{L^2(M \setminus \Sigma; dV)}
= 
(\Delta_{g, \mu} u^f(t), \phi)_{L^2(M \setminus \Sigma; dV)} 
\\\nonumber&\quad= 
- (\Grad u^f(t), \Grad \phi)_{L^2(M \setminus \Sigma; dV)}
+ (\p_{\nu, \mu} u^f(t), \phi)_{L^2(\p M; dS_g)},
\end{align}
where $dS_g$ denotes 
the Riemannian surface measure on $(\p M, g)$.
Notice that the boundary term on $\p \Sigma$ vanish as $u^f$ satisfies the homogeneous Neumann boundary condition there.

In particular, for $f, h \in C_0^\infty((0,T) \times \Gamma)$,
$t \in (0, T)$ and $s \in (0, 2T)$,
\begin{align*}
&(\p_t^2 - \p_s^2) (u^f(t), u^h(s))_{L^2(M \setminus \Sigma; dV)}
\\&\quad= 
(f(t), \Lambda_{2T, \Gamma} h(s))_{L^2(\p M; dS_g)} 
- (\Lambda_{T, \Gamma} f(t), h(s))_{L^2(\p M; dS_g)}.
\end{align*}
By solving this wave equation with vanishing initial conditions at $t = 0$ 
and noticing that $\Lambda_{T, \Gamma}^* = R \Lambda_{T, \Gamma} R$, 
we get the Blagove{\v{s}}{\v{c}}enski{\u\i}'s identity
\begin{align}
\label{eq_Blago}
(u^f(T), u^h(T))_{L^2(M \setminus \Sigma; dV)}
&= (f, K h)_{L^2((0,T) \times \Gamma; dt \otimes dS_g)},
\end{align}
that holds for all $f, h \in L^2((0, T) \times \Gamma)$
by continuity of $K$ and density of smooth functions in $L^2$.
The identity (\ref{eq_Blago}) originates from \cite{Blagovevsvcenskiui1966}.

Moreover, by letting $\phi = 1$ identically in (\ref{the_integration_by_parts}), we get
\begin{align}
\label{eq_diff_for_1}
&\p_t^2 (u^f(t), 1)_{L^2(M \setminus \Sigma; dV)}
= (\p_{\nu, \mu} u^f(t), 1)_{L^2(\p M; dS_g)}.
\end{align}
Notice that this identity does not hold if $u^f$ satisfies 
the homogeneous Dirichlet boundary condition on $\p \Sigma$,
instead of the Neumann one. 
This is why our method does not extend 
to detection of sound soft obstacles in a straightforward manner. 
We get the indentity
\begin{align}
\label{eq_Blago_1}
(u^f(T), 1)_{L^2(M \setminus \Sigma; dV)}
= (f, b)_{L^2((0,T) \times \Gamma; dt \otimes dS_g)},
\end{align}
where $b(t, x) = T - t$,
by solving the ordinary differential equation (\ref{eq_diff_for_1}) with vanishing initial conditions at $t = 0$.

Let $\tau \in C_T(\Gamma)$ and let us define the set
\begin{align*}
S_\tau := \{(t, x) \in [0,T] \times \bar \Gamma;\ t \in [T - \tau(x), T] \}.
\end{align*}
We define the operator
\begin{align*}
W_\tau f := u^f(T), \quad
W_\tau : L^2(S_\tau) \to L^2(M \setminus \Sigma).
\end{align*}
It follows from \cite{Lasiecka1991} that $W_\tau$ is compact. 
Moreover, we may consider a restriction of $K$,
\begin{align*}
K_\tau f = K f|_{S_\tau}, \quad K_\tau : L^2(S_\tau) \to L^2(S_\tau).
\end{align*}
Then the equations (\ref{eq_Blago}) and (\ref{eq_Blago_1}) yield
that on $L^2(S_\tau)$ 
\begin{align}
\label{from_boundary_to_interior}
W_\tau^* W_\tau = K_\tau, \quad W_\tau^* 1 = b.
\end{align}

Let us now consider the control equation,
\begin{align}
\label{eq_control}
W_\tau f = 1, \quad \text{for $f \in L^2(S_\tau)$}.
\end{align}
We have $\supp(W_\tau f) \subset M_\Sigma(\tau)$
since the wave equation (\ref{eq_wave}) has finite speed of propagation.
Moreover, it can be shown using Tataru's unique continuation \cite{Tataru1995}
that the inclusion 
\begin{align}
\label{density}
\{W_\tau f;\ f \in L^2(S_\tau)\} \subset L^2(M_\Sigma(\tau)), 
\end{align}
is dense, see the appendix below. 
In particular, if there is a least squares solution $f_0$ to (\ref{eq_control})
then $W_\tau f_0 = 1_{M_\Sigma(\tau)}$.
However, as $W_\tau$ is compact, the range of $W_\tau$ 
is a proper dense subset of $L^2(M_\Sigma(\tau))$
and (\ref{eq_control}) may fail to have a least squares solution.
Nonetheless, it is instructive to consider first the case 
where (\ref{eq_control}) has a least squares solution.
Then the least squares solution of minimal norm $f_0$ is given by the pseudoinverse,
see e.g. \cite[Th. 2.6]{Engl1996},
\begin{align*}
f_0 = W_\tau^\dagger 1 = (W_\tau^* W_\tau)^\dagger W_\tau^* 1 
= K_\tau^\dagger b,
\end{align*}
and we can compute the volume $V(M_\Sigma(\tau))$ from the boundary data 
$\Lambda_{2T, \Lambda}$ by the formula
\begin{align*}
V(M_\Sigma(\tau)) &= 
(1_{M_\Sigma(\tau)}, 1)_{L^2(M \setminus \Sigma; dV)}
= (W_\tau W_\tau^\dagger 1, 1)_{L^2(M \setminus \Sigma; dV)}
\\&= (K_\tau^\dagger b, b)_{L^2(S_\tau; dt \otimes dS_g)}.
\end{align*}

%

The standard technique to remedy the nonexistence of a least squares solution to 
a linear equation is to use a regularization method.
As $W_\tau$ is compact and we have the information (\ref{from_boundary_to_interior})
at our disposal, there are several ways to regularize that are
available to us. For example, we could use a regularization by projection 
\cite[Section 3.3]{Engl1996}
or a regularization based on a spectral approximation of the inverse \cite[Th. 4.1]{Engl1996}.
Here we will consider only the classical Tikhonov regularization,
\begin{align}
\label{tikhonov}
f_\alpha := (W_\tau^* W_\tau + \alpha)^{-1} W_\tau^* 1 = (K_\tau + \alpha)^{-1} b,
\quad \alpha > 0.
\end{align}

We have the following abstract lemma.

\begin{lemma}
Suppose that $X$ and $Y$ are Hilbert spaces.
Let $y \in Y$ and let $A : X \to Y$ be a bounded linear operator 
with the range $R(A)$. 
Then $A x_\alpha \to Py$ as $\alpha \to 0$, 
where $x_\alpha = (A^* A + \alpha)^{-1} A^* y$, $\alpha > 0$,
and $P : Y \to \overline{R(A)}$ is the orthogonal projection.
\end{lemma}
\begin{proof}
Notice that for all $x \in X$
\begin{align*}
\norm{A x- y}^2 = \norm{A x - P y}^2 + \norm{(1 - P)y}^2.
\end{align*}
By \cite[Th. 5.1]{Engl1996} we know that $x_\alpha$ is the unique minimizer of 
\begin{align*}
\norm{A x - y}^2 + \alpha \norm{x}.
\end{align*}
Let $\epsilon > 0$ and let $x^\epsilon \in X$ satisfy
$\norm{A x^\epsilon - Py}^2 < \epsilon$.
Then
\begin{align*}
\norm{A x_\alpha - P y}^2 &= \norm{A x_\alpha- y}^2 - \norm{(1 - P)y}^2
\\&\le \norm{A x_\alpha - y}^2 + \alpha \norm{x_\alpha} - \norm{(1 - P)y}^2
\\&\le \norm{A x^\epsilon - y}^2 + \alpha \norm{x^\epsilon} - \norm{(1 - P)y}^2
\\&= \norm{A x^\epsilon - Py}^2 + \alpha \norm{x^\epsilon}
< \epsilon + \alpha \norm{x^\epsilon}
\le 2 \epsilon,
\end{align*}
for $\alpha \le \epsilon / \norm{x^\epsilon}$.
\end{proof}

By the density (\ref{density}) we have that $\overline{R(W_\tau)} = L^2(M_\Sigma(\tau))$.
Thus the above lemma implies that 
$W f_\alpha \to 1_{M_\Sigma(\tau)}$ in $L^2(M \setminus \Sigma)$ as
$\alpha$ tends to zero.
In particular, we may compute the volume $V(M_\Sigma(\tau))$ from the boundary data 
$\Lambda_{2T, \Lambda}$ by the formula
\begin{align}
\label{eq_reconstruction_vol}
V(M_\Sigma(\tau)) = \lim_{\alpha \to 0+} 
((K_\tau + \alpha)^{-1} b, b)_{L^2(S_\tau; dt \otimes dS_g)}.
\end{align}

\begin{lemma}
\label{lem_obstacle_from_volumes}
Let $T > 0$, $\Gamma \subset \p M$ be open and let $\tau \in C_T(\Gamma)$.
Then 
\begin{align*}
M(\tau)^\inter \cap \Sigma^\inter \ne \emptyset
\quad \text{if and only if} \quad
V(M_\Sigma(\tau)) < V(M(\tau)).
\end{align*}
\end{lemma}
\begin{proof}
Notice that $d_\Sigma(x, y) \ge d(x, y)$ for any $x, y \in M \setminus \Sigma$.
Hence $M_\Sigma(\tau) \subset M(\tau)$.
Morever, $M_\Sigma(\tau) \cap \Sigma = \emptyset$ by definition. 
In particular, if the open set $M(\tau)^\inter \cap \Sigma^\inter$ is nonempty, then 
\begin{align*}
V(M_\Sigma(\tau)) &\le V(M(\tau) \setminus \Sigma) 
\\&< V(M(\tau) \setminus \Sigma) + V(M(\tau)^\inter \cap \Sigma^\inter)
\le V(M(\tau)).
\end{align*}
Thus we have shown the implication from left to right in (\ref{eq_detection}).

Let us now suppose that $V(M_\Sigma(\tau)) < V(M(\tau))$.
Then $M(\tau) \setminus M_\Sigma(\tau)$ is not a null set (that is, a set of measure 0).
But $\p M(\tau)$ is a null set \cite{Oksanen2011}, whence 
there is $x \in M(\tau)^\inter \setminus M_\Sigma(\tau)$.
Thus there is $y \in \p M$ and a path $\gamma : [0, \ell] \to M$
from $y$ to $x$ such that the length of $\gamma$ satisfies $l(\gamma) \le \tau(y)$.
The path $\gamma$ intersects $\Sigma$ since otherwise we would have $x \in M_\Sigma(\tau)$.
Let $z \in \Sigma \cap \gamma([0, \ell])$.
Then $z \in M(\tau)^\inter$ since $x \in M(\tau)^\inter$, and
there is a neighborhood $U \subset M(\tau)^\inter$ of $z$ such that $U \cap \Sigma^\inter \ne \emptyset$.
Hence also $M(\tau)^\inter \cap \Sigma^\inter \ne \emptyset$.
\end{proof}

Theorem \ref{thm_main} follows from the formula (\ref{eq_reconstruction_vol}) 
and Lemma \ref{lem_obstacle_from_volumes}.

\section{Numerical results}
\label{sec_computations}
\def\SNR{\text{SNR}}

\subsection{Simulation of the data}

In all our numerical examples $(M, g)$ is the two-dimensional 
unit square with the Euclidean metric, that is,
\begin{align*}
M = [0,1]^2, 
\quad g = (dx^1)^2 + (dx^2)^2.
\quad 
\end{align*}
Moreover, $T = 1$ and the accessible part of the boundary $\Gamma$ is the bottom edge of $M$,
\begin{align*}
\Gamma = \{(x^1, 0) \in M; x^1 \in (0,1) \}.
\end{align*}

For computation of the Dirichlet-to-Neumann map we discretize in space by 
using finite elements, and solve the resulting system of ordinary differential equations by 
a backward differentiation formula (BDF).
To be very specific, we use the commercial Comsol solver
with quadratic Lagrange elements and BDF time-stepping with maximum order of 2.
Both the maximum element size and time step size are set to the constant value 
$h = 0.0025$.

We discretize the 
measurement $\Lambda_{\Gamma, 2T} f$, $f \in L^2((0, T) \times \Gamma)$,
by taking the point values on the uniform grid of 
temporal points $t_j \in [0, 2T]$, $j=1,2,\dots,N_t$,
and spatial points $x_k \in \Gamma$, $k=1,2,\dots,N_x$,
where $N_x = 20$ and $N_t = 800$.
The higher precision in time reflects the fact that 
a measurement of this type can realized by using $N_x$
receivers (e.g. microphones) with the sampling rate $h$.

We model noisy measurements by adding white Gaussian noise 
to the signal
\begin{align*}
\lambda_f(j, k) := \Lambda_{\Gamma, T} f(t_j, x_k), 
\quad j = 1,2, \dots, N_t,\ k=1,2,\dots, N_x.
\end{align*}
To be very specific, we use the Matlab function \verb+awgn+
both to measure the power of the signal $\lambda_f$
and to add noise with specified signal-to-noise ratio ($\SNR$).
We have used signal-to-noise ratios  $14 dB$
and $7 dB$ corresponding to $4\%$ and $20\%$ noise power levels.

\subsection{Solving the control equation}
\def\spC{\mathcal C}

The operator $K_\tau$ is self-adjoint and positive-semidefinite 
by (\ref{from_boundary_to_interior}),
whence $K_\tau + \alpha$ positive-definite for $\alpha > 0$.
We solve the Tikhonov regularized control equation 
\begin{align}
\label{eq_control_K}
(K_\tau + \alpha)f = b
\end{align}
by using the conjugate gradient (CG) method on 
a finite dimensional subspace $\spC_\tau \subset L^2(S_\tau)$
that we will define below. We have used the initial value $f = 0$
in all our CG iterations. 

We denote by $\Gamma_k \subset \Gamma$ the Voronoi cell corresponding 
to the measurement point $x_k$, $k=1,2,\dots, N_x$, that is, 
\begin{align*}
\Gamma_k := \{x \in \Gamma;\ |x - x_k| \le |x - x_l|,\ l = 1,2, \dots, N_x \}.
\end{align*}
Moreover, we denote by $\spC$ the space of piecewise constant sources $f$
that can be represented as a linear combination of the functions
\begin{align}
\label{def_pulse}
f_k(t, x) := 1_{[0, h]}(t) 1_{\Gamma_k}(x), \quad k = 1,2, \dots, N_x,
\end{align}
and their time translations by an integer multiple of $h$.
Finally, we define
\begin{align*}
\spC_\tau := \{ f \in \spC;\ \supp(f) \subset S_\tau \},
\quad \tau \in C_T(\Gamma).
\end{align*}

As the wave equation (\ref{eq_wave}) is invariant with respect to 
translations in time, we can compute $\lambda_f$ for 
arbitrary $f \in \spC_\tau$ and $\tau \in C_T(\Gamma)$ 
if we are given the measurements 
\begin{align*}
\lambda_{f_k}, \quad k = 1,2, \dots, N_x.
\end{align*}
To summarize, we employ $N_x = 20$ measurements that can be 
realized by using $N_x$ receivers 
with the sampling rate $h = 0.0025$.

\subsection{Regularization and calibration}

\begin{figure}[t]
\centering
\includegraphics[scale=0.6]
{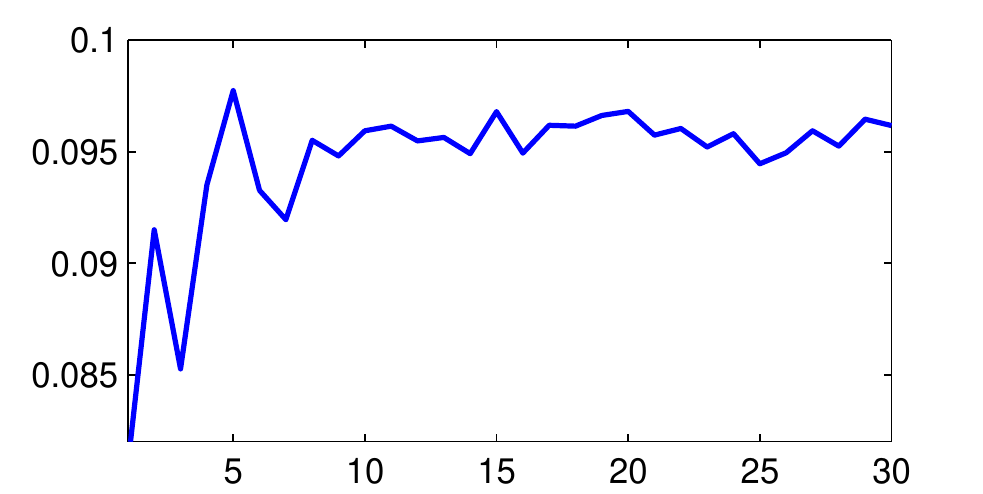}
\includegraphics[scale=0.6]
{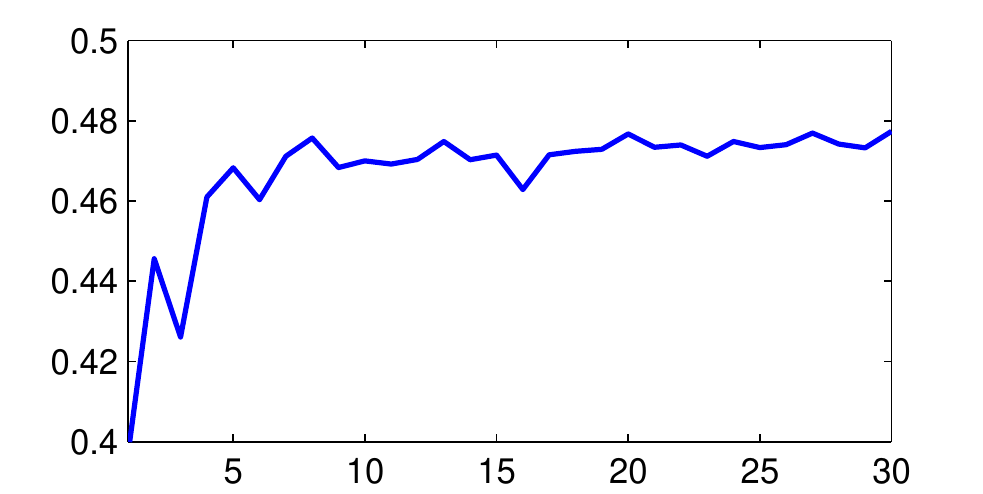}
\caption{
Convergence of the reconstructed volume $V(M_\emptyset(\tau_r))$
as a function of $N_{cg}$, i.e. the number of the conjugate gradient steps. 
Noiseless case.
{\em Left:} $r = 1/10$; {\em right:} $r = 1/2$.
}
\label{fig_convergence}
\end{figure}

As the control equation (\ref{eq_control_K}) may be ill-posed
for $\alpha = 0$, we terminate the CG iteration early after $N_{cg}$ steps. 
This amounts to regularization of the problem \cite{Hanke1995}.
To calibrate the method we probed the empty space case, $\Sigma = \emptyset$,
with half-spaces. That is, we chose the profile function $\tau \in C_T(\Gamma)$
to be of the form, 
\begin{align*}
\tau_r(x) := r, \quad x \in \Gamma,\quad r \in [1/10, 1/2].
\end{align*}
In this case, the CG iteration essentially converges after 10 steps
even when not using the Tikhonov regularization, that is, $\alpha = 0$,
see Figure \ref{fig_convergence}.
For this reason, we have chosen $N_{cg} = 10$ in all our further simulations.

In addition to the empty space case, 
we have experimented with the disk and the square shaped 
obstacles defined as follows:
$\Sigma_\circ$ is the disk with radius $3/10$ 
and center $p := (1/2, 1/2)$
and $\Sigma_\diamond$ is the square with side length $0.424$,
center $p$ and axes rotated by $\pi/4$ with respect to the axes of $M$,
see Figure \ref{fig_geom}.

\begin{figure}[t]
\centering
\includegraphics[scale=0.6]
{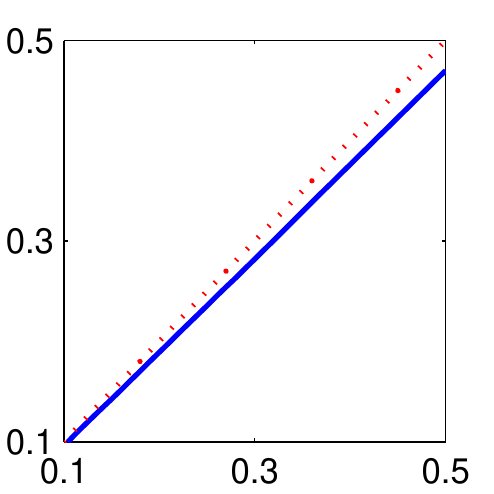}
\includegraphics[scale=0.6]
{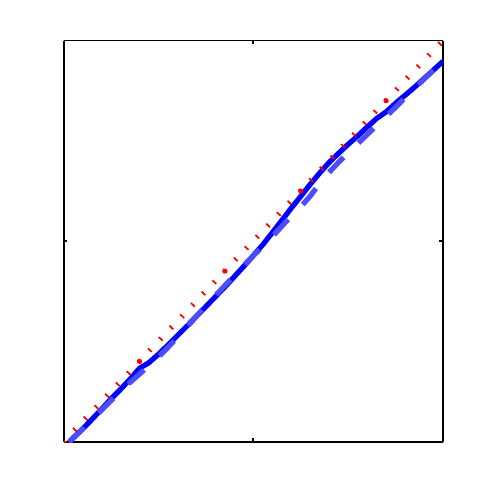}
\includegraphics[scale=0.6]
{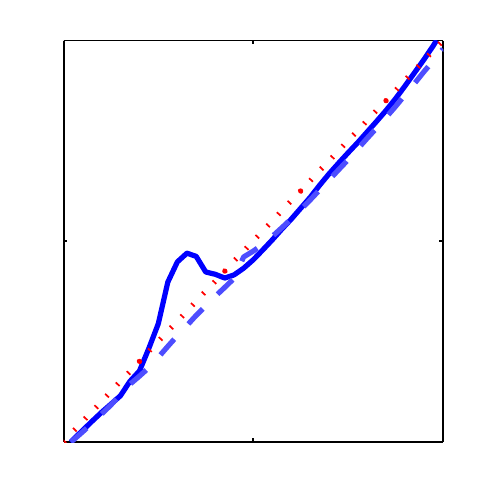}
\includegraphics[scale=0.6]
{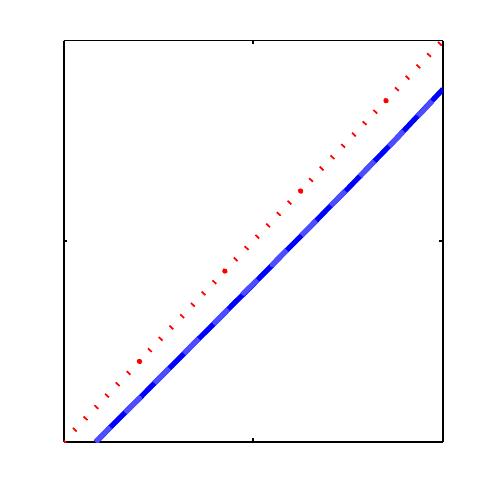}
\caption{
Reconstructed volumes $V(M_\emptyset(\tau_r))$  
as a function of $r$ compared to the real volume (dotted red).
The two reconstructions (solid blue and dashed blue) correspond to two
different realizations of noise.
{\em From left, 1st:} 
noiseless, $\alpha = 0$;
{\em 2nd:} 
$\SNR = 14 dB$, $\alpha = 0$;
{\em 3rd:} 
$\SNR = 7 dB$, $\alpha = 0$;
{\em 4th:} 
$\SNR = 7 dB$, $\alpha = 10^{-3}$.
}
\label{fig_vols_empty}
\end{figure}

\begin{figure}[t]
\centering
\includegraphics[scale=0.6]
{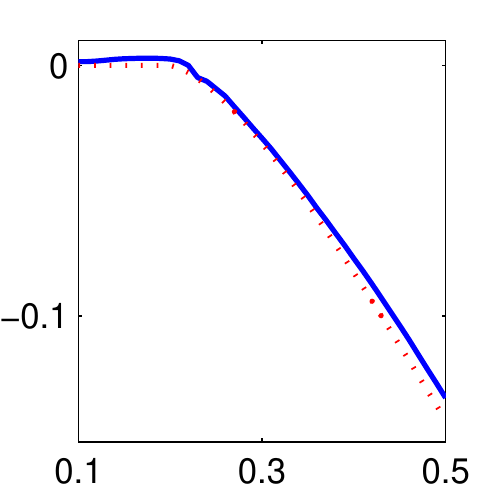}
\includegraphics[scale=0.6]
{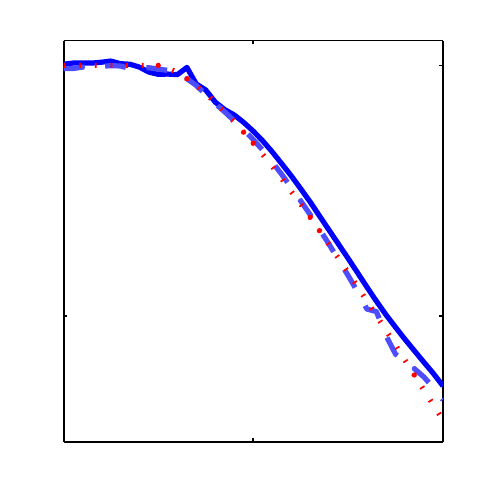}
\includegraphics[scale=0.6]
{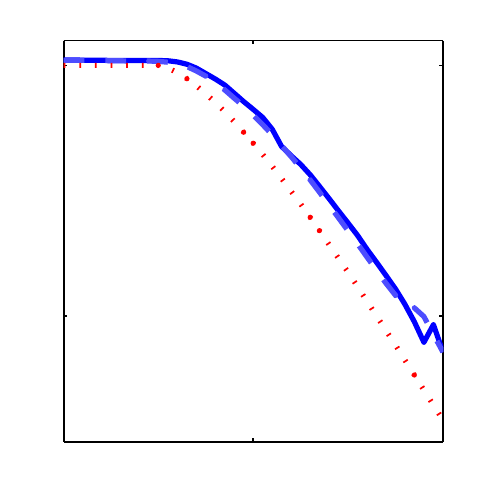}
\caption{
Reconstructed volume differences (\ref{vol_diffs}) with $\tau = \tau_r$
as a function of $r$
compared to the real difference (dotted red)
in the case of the disk shaped obstacle $\Sigma = \Sigma_\circ$.
The two reconstructions (solid blue and dashed blue) correspond to two
different realizations of noise.
{\em From left, 1st:} 
noiseless, $\alpha = 0$;
{\em 2nd:} 
$\SNR = 14 dB$, $\alpha = 0$;
{\em 3rd:} 
$\SNR = 7 dB$, $\alpha = 10^{-3}$.
}
\label{fig_vols_disk}
\end{figure}

It is not clear to us, why the method underestimates the volume $V(M_\emptyset(\tau_r))$,
see Figure \ref{fig_vols_empty} (leftmost plot).
One possibility is that we using too few spatial basis functions,
however, the smallness of $N_x$ is motivated by applications.
Moreover, the underestimation is systematic and is canceled
when considering the volume differences,
\begin{align}
\label{vol_diffs}
V(M_\Sigma(\tau)) - V(M_\emptyset(\tau)),
\end{align}
see Figure \ref{fig_vols_disk}.
In terms of applications, this means that we should calibrate the method
in a known background before probing a region that possibly contains an obstacle.

According to our experiments the method reconstructs volumes 
reliably when $\SNR=14 dB$ and we regularize only 
through the early termination of the CG iteration. 
When $\SNR=7 dB$ and $\alpha = 0$, a reconstruction 
can be seriously disrupted even in the empty space case.
After introducing Tikhonov regularization with $\alpha = 10^{-3}$,
the effect of noise vanishes but a large systematic error 
appears, see Figure \ref{fig_vols_empty} (the two rightmost plots).
We see that considering the volume differences (\ref{vol_diffs})
becomes even more essential when $\alpha > 0$.

\subsection{Probing with disk shaped domains of influence}

\begin{figure}[t]
\centering
\includegraphics[scale=0.6]
{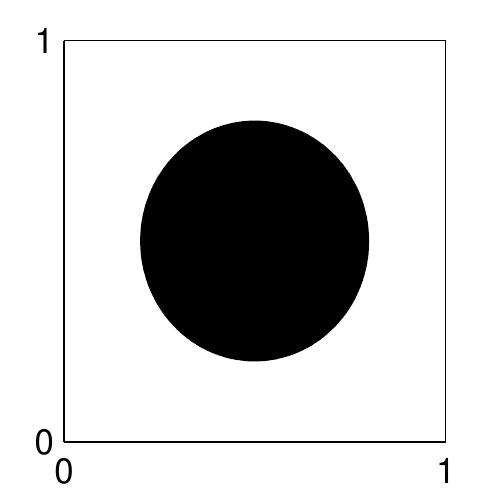}
\includegraphics[scale=0.6]
{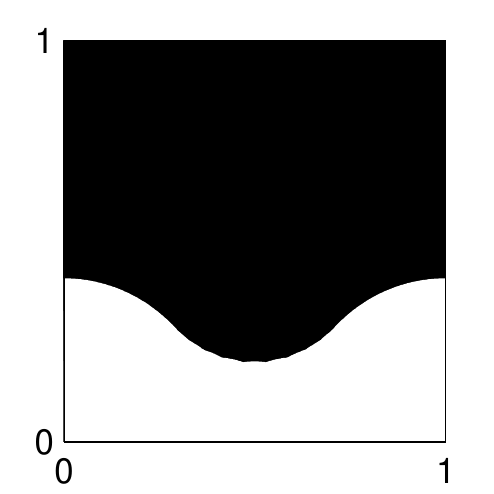}
\includegraphics[scale=0.6]
{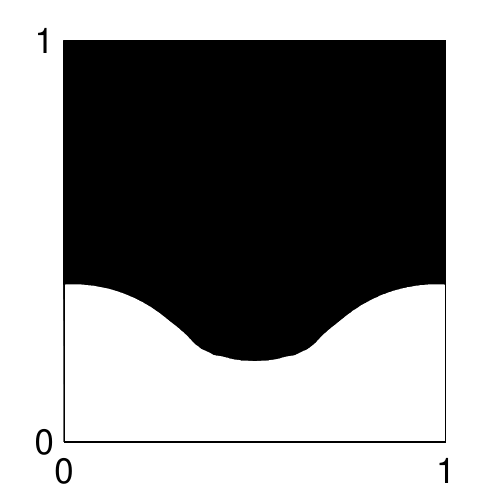}
\\
\includegraphics[scale=0.6]
{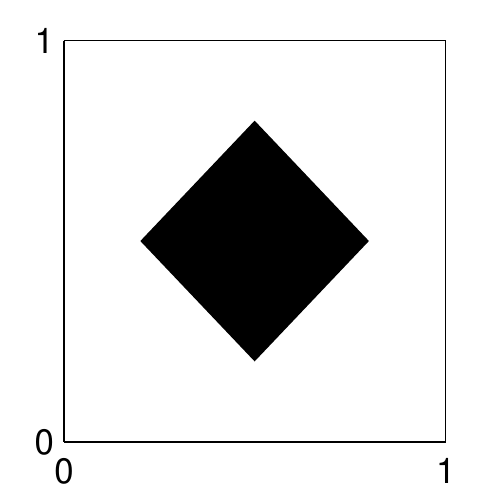}
\includegraphics[scale=0.6]
{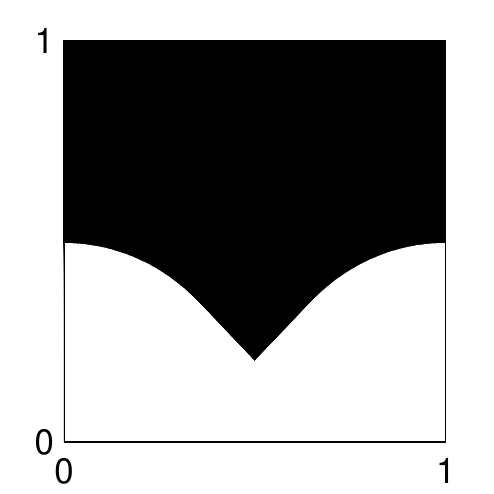}
\includegraphics[scale=0.6]
{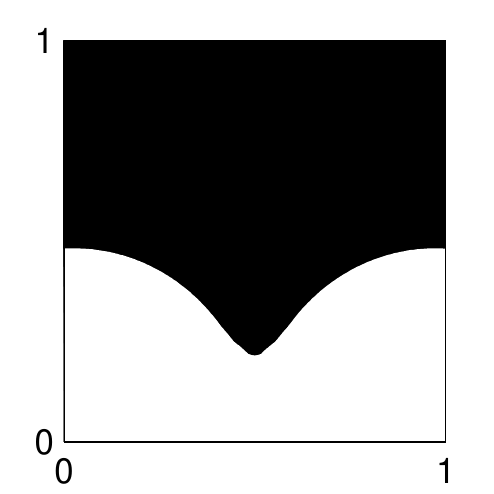}
\caption{
{\em Top row from left, 1st:} 
The disk shaped obstacle $\Sigma = \Sigma_\circ$.
{\em 2nd:} 
The largest region on which the absence of the obstacle 
can be concluded when probing with disk shaped 
domains of influence $H_{\Sigma_\circ}(\Gamma)$.
{\em 3rd:} 
A reconstruction of $H_{\Sigma_\circ}(\Gamma)$ in the noiseless case. 
The threshold $\epsilon = 5/10^{4}$.
{\em Bottom row:} 
The square shaped obstacle $\Sigma = \Sigma_\diamond$,
$H_{\Sigma_\diamond}(\Gamma)$
and a reconstruction of $H_{\Sigma_\diamond}(\Gamma)$.
The same parameter values are used for both the reconstructions.
}
\label{fig_geom}
\end{figure}

We will now describe our experiments concerning reconstruction of the shape of an obstacle. To this purpose, we chose the profile function $\tau \in C_T(\Gamma)$
to be of the form, 
\begin{align*}
\tau_r^y(x) := r - |x - y|, \quad x \in \Gamma,\quad y \in \Gamma,\ r \in [1/10, 1/2].
\end{align*}
Then $M(\tau_r^y) = \overline{B(y, r)} \cap M$, that is, the intersection of $M$ and
the closed disk of radius $r$ centered at $y$.
Probing with disks has been considered in the context of 
electrical impedance tomography in \cite{Ide2007}
and our numerical results are comparable to the results therein. 

Analogously to \cite{Ide2007} and \cite{Oksanen2011a},
let us define the largest region $H_\Sigma(\Gamma)$ on which
we can conclude the absence of obstacles by probing with the sets 
$\overline{B(y, r)} \cap M$, $y \in \Gamma$, $r \in (0, T]$.
We denote 
\begin{align*}
R_T(y) :&= \sup \{ r \in (0, T];\ B(y, r) \cap \Sigma^\inter = \emptyset \}
\\& = \sup \{ r \in (0, T];\ V(M_\Sigma(\tau_r^y)) = V(M(\tau_r^y))\},
\end{align*}
and define
\begin{align*}
H_\Sigma(\Gamma) := \bigcup_{y \in \Gamma} \ll(B(y, R_T(y)) \cap M \rr).
\end{align*}

\begin{figure}[t]
\centering
\includegraphics[scale=0.6]
{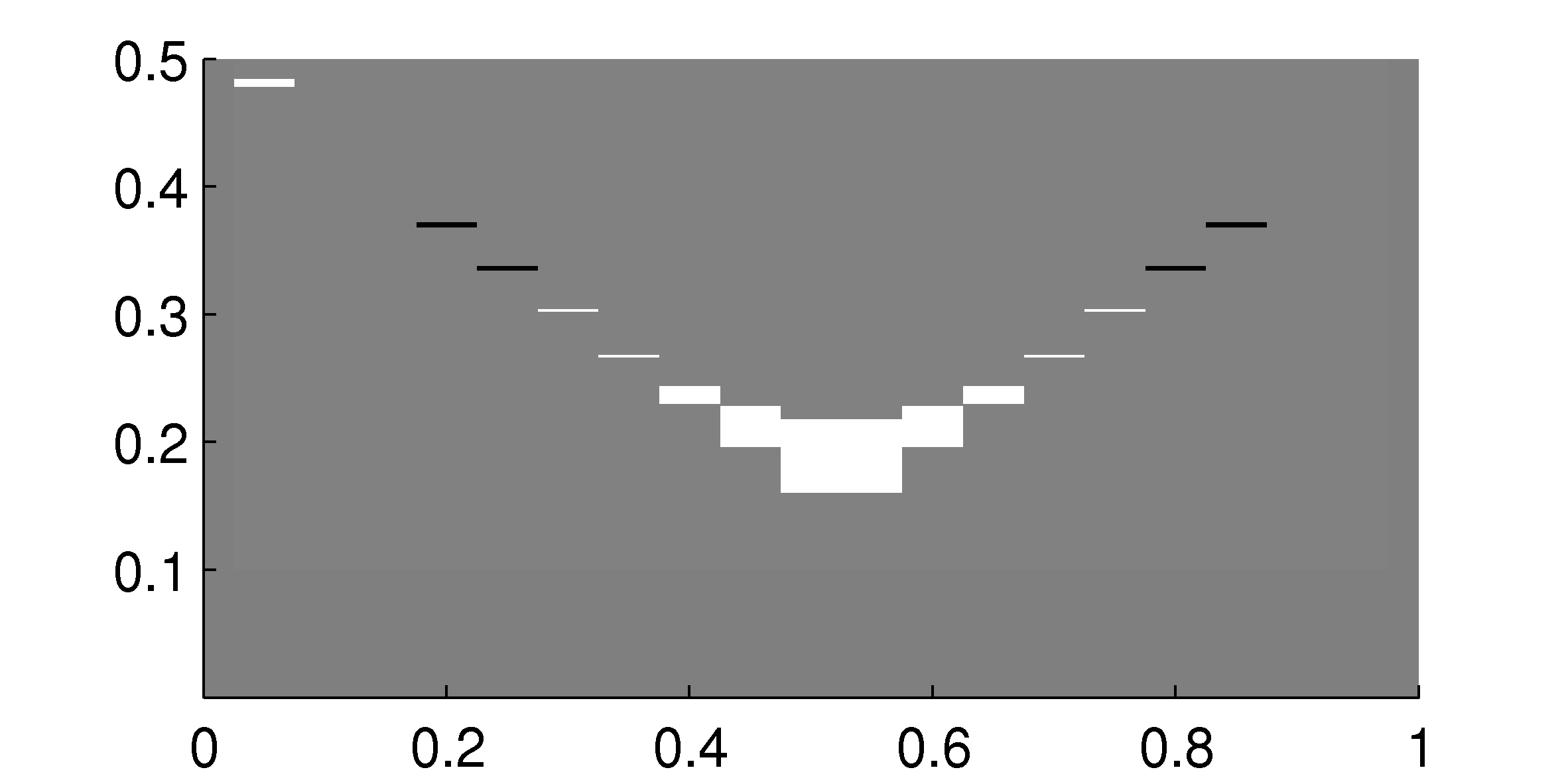}
\includegraphics[scale=0.6]
{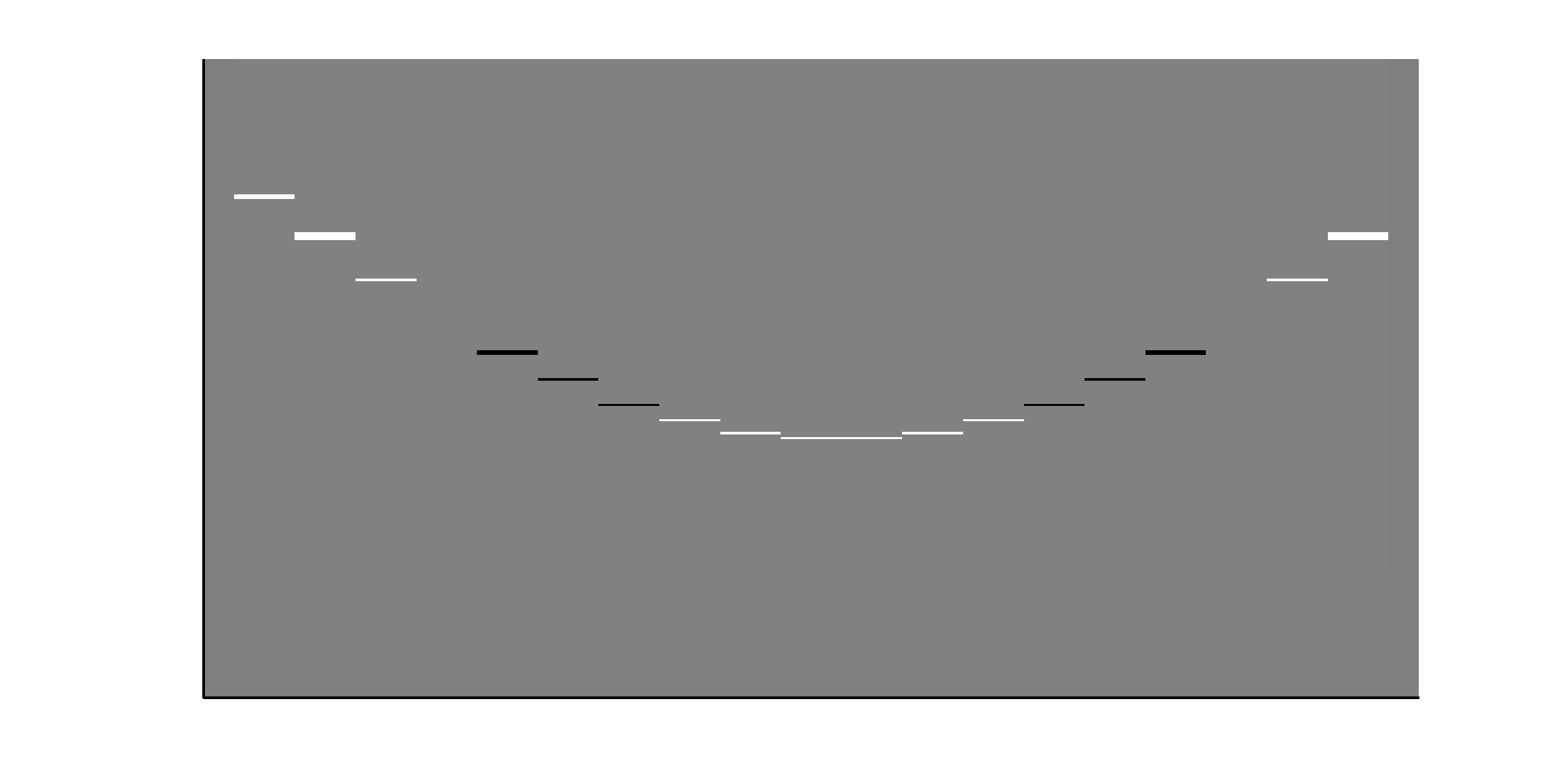}
\\
\includegraphics[scale=0.6]
{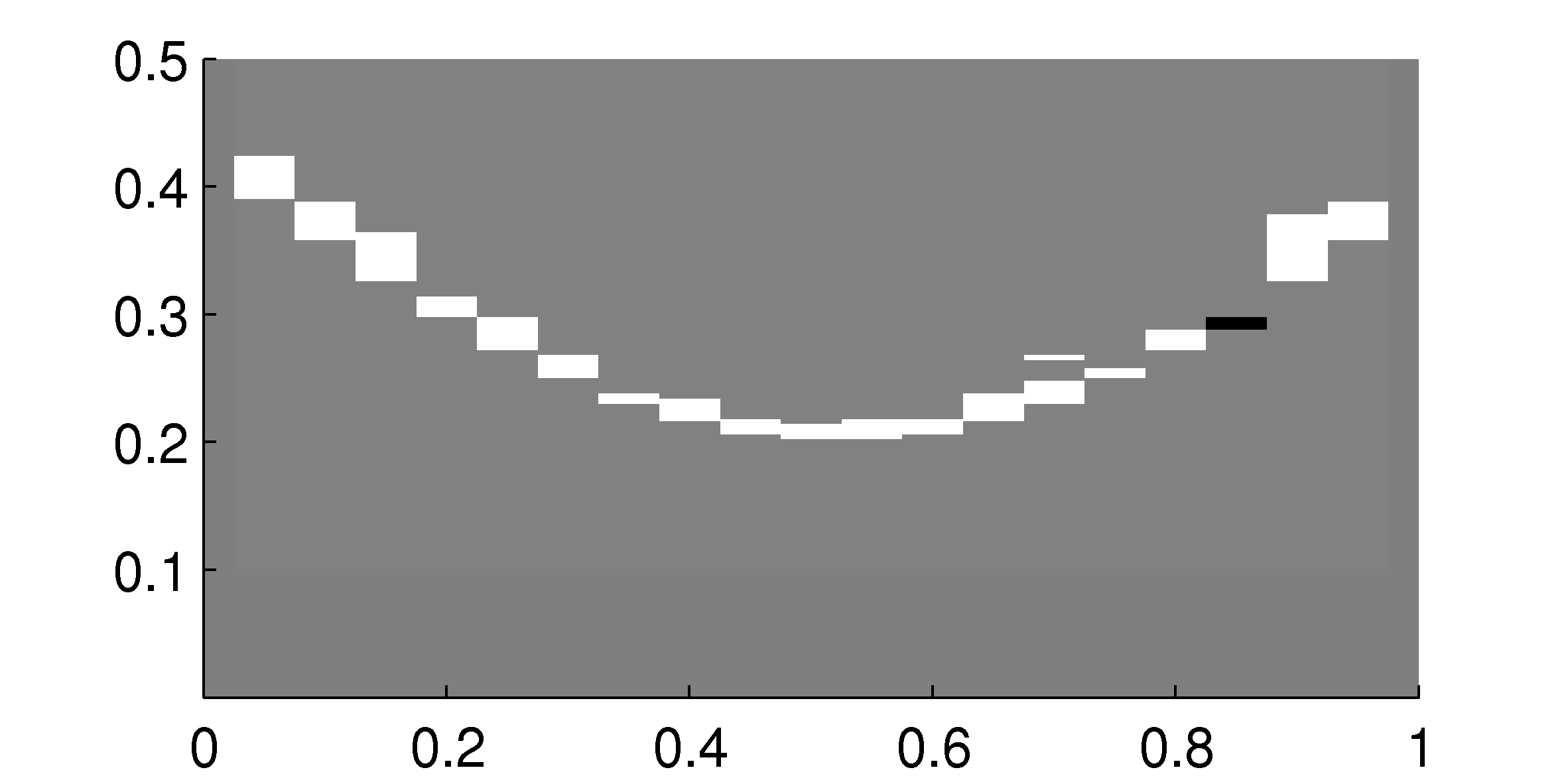}
\includegraphics[scale=0.6]
{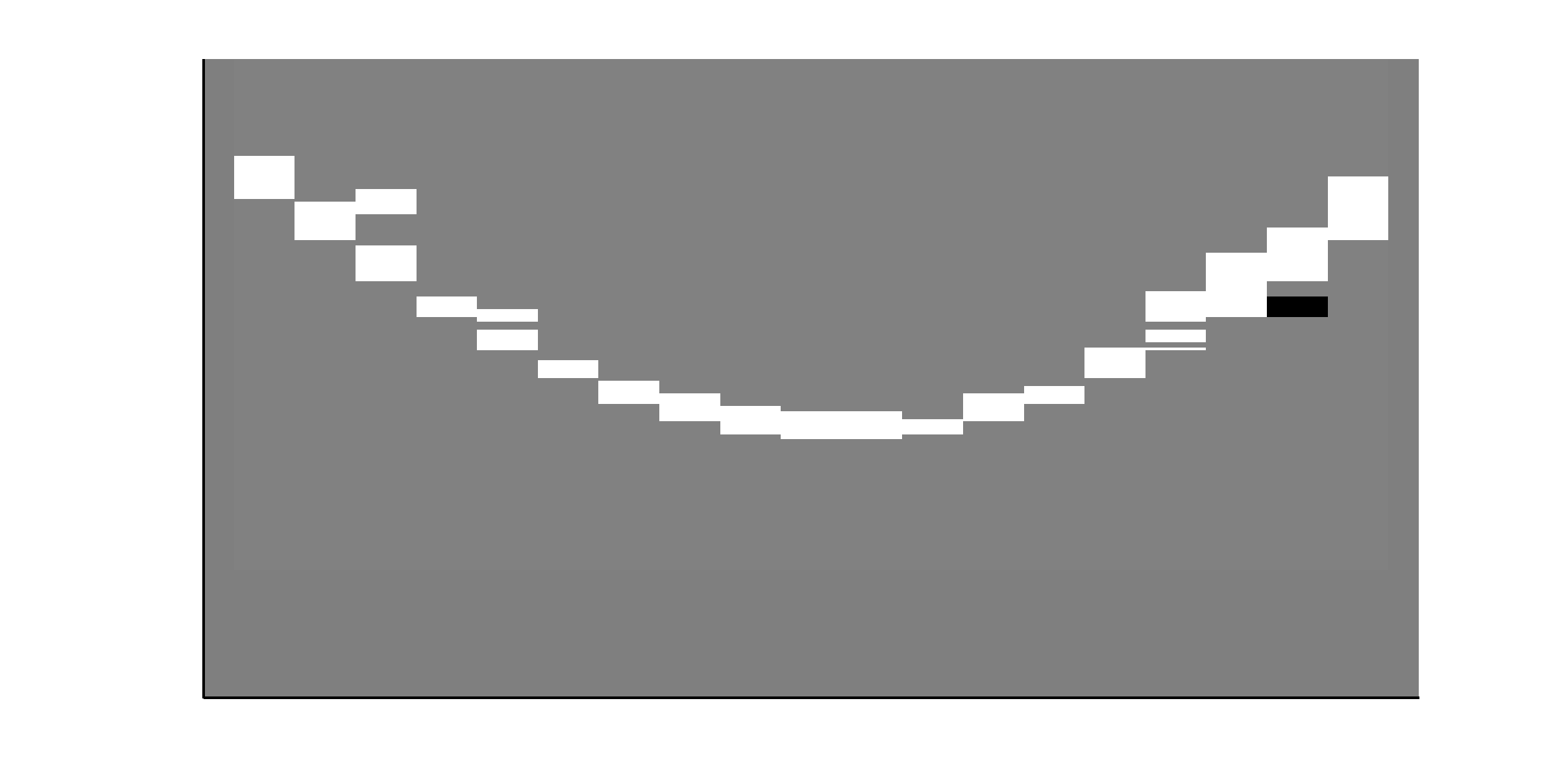}
\caption{
Comparison of a reconstruction $\tilde H_\Sigma(\Gamma)$ with the 
theoretical best possible reconstruction $H_\Sigma(\Gamma)$.
Erroneous pixels in $\tilde H_\Sigma(\Gamma) \setminus H_\Sigma(\Gamma)$
are drawn in white
and $H_\Sigma(\Gamma) \setminus \tilde H_\Sigma(\Gamma)$
in black.
Gray pixels are reconstructed correctly.
{\em Top row:} Noiseless measurements with the threshold $\epsilon = 5/10^{4}$,
{\em left:} $\Sigma = \Sigma_\diamond$;
{\em right:} $\Sigma = \Sigma_\circ$.
{\em Bottom row:} $\SNR=14dB$, $\Sigma = \Sigma_\circ$ and
$\epsilon = 4/10^{3}$.
The two reconstructions correspond to two
different realizations of noise.
}
\label{fig_rec_diff}
\end{figure}

Let us describe next how we approximate $R_T(y)$
when computing with finite precision. 
Let $\epsilon > 0$, $N_r \in \N$ 
and let $r_l \in [0, T]$, $l = 1, 2, \dots, N_r$, be a uniform grid of points.
We denote
\begin{align*}
L(\epsilon, N_r) := \max \{l = 1, 2, \dots, N_r;\ 
V(M_\Sigma(\tau_r^y)) - V(M_\emptyset(\tau_r^y)) \ge -\epsilon \},
\end{align*}
and define the approximation $r_T(y; \epsilon, N_r) = r_{L(\epsilon, N_r)}$
of $R_T(y)$.
We have used the threshold $\epsilon = 5/10^4$ in noiseless cases
and $\epsilon = 4/10^3$ when $\SNR = 14 dB$.
According to our numerical experiments the method reconstructs 
$H_\Sigma(\Gamma)$ reliably when using these values of $\epsilon$
and $N_r = 500$, see Figure \ref{fig_rec_diff}, where
a white pixel means that the center point of the pixel is 
erroneously identified to be in $H_\Sigma(\Gamma)$ (false positive)
and a black pixel means erroneously identification of not being in $H_\Sigma(\Gamma)$
(false negative).

Computationally the shape reconstruction amounts to solving 
a large number of independent systems of linear equations
by running a few number of CG steps for each of them.
Our implementation with parameters as above and $r_l$'s restricted in $[1/10, 1/2]$ 
led to 4020 systems with the number of unknowns varying between 
30 and 1000. The run time for the full reconstruction on a single processor was about 10 minutes, however, as the systems are independent, the method allows for an efficient parallel implementation. 

\section*{Appendix: Approximate controllability}

In this section we show that the inclusion (\ref{density}) is dense, that is
we prove the following lemma.

\begin{lemma}
\label{lem_uniq_cont}
Let $T > 0$, let $\Gamma \subset \p M$ be open
and let $\tau \in C_T(\Gamma)$. 
Then 
\begin{align}
\label{the_dense_set}
\{ u^f(T);\ f \in C_0^\infty(S_\tau) \}
\end{align}
is dense in $L^2(M_\Sigma(\tau))$.
\end{lemma}
A density result of this type is called approximate controllability
in the control theoretic literature.
To our knowledge, Lemma \ref{lem_uniq_cont} is proved previously only 
in the case of a constant function $\tau$, see e.g. \cite[Th. 3.10]{Katchalov2001}.
We will give a proof in the general case $\tau \in C_T(\Gamma)$ 
by reducing it to the constant function case.
To simplify the notation we consider only the case $\Sigma = \emptyset$,
since the general case follows by replacing $M$ by $M \setminus \Sigma$
in the proofs below. 

\begin{lemma}
\label{lem_uniq_cont_induction}
Let $T > 0$, $J \in \N$, let $\Gamma_j \subset \p M$ be open 
and let $h_j \in C_T(\Gamma_j)$ for $j = 1, 2, \dots, J$.
We define
\begin{align}
\label{max_h}
h^J(y) := 
\begin{cases}
\max\{ h_j(y); \text{$j$ satisfies $\bar \Gamma_j \ni y$} \},
& y \in \bigcup_{j=1}^J \bar \Gamma_j,
\\
0, &\text{otherwise}.
\end{cases}
\end{align}
If for all $j=1, \dots, J$ 
the functions (\ref{the_dense_set}) for $\tau = h_j$ are dense 
in $L^2(M(h_j))$, 
then the functions (\ref{the_dense_set}) for $\tau = h^J$
are dense in $L^2(M(h^J))$. 
\end{lemma}
\begin{proof}
Notice that $\p M \subset M(\tau)$ if $\tau(y) \ge 0$
for all $y \in \p M$.
Abusing the notation slightly, we will consider $M(\tau)$
as a subset of $M^\inter$. 
This does not affect the density since $\p M$ is a null set. 
We denote $\Gamma^J := \bigcup_{j=1}^J \Gamma_j$
and have
\begin{align*}
M(h^J) 
&= 
\{x \in M^\inter;\ \text{there is $y \in \bar \Gamma^J$ s.t. $d(x, y) \le h(y)$}\}
\\&= 
\bigcup_{j=1}^J \{x \in M^\inter;\ \text{there is $y \in \bar \Gamma_j$ s.t. $d(x, y) \le h_j(y)$}\}
\\&=
\bigcup_{j=1}^J  M(\Gamma_1, h_j).
\end{align*}

We will now prove the density by induction with respect to $J$.
The case $J = 1$ is trivial.
Let us denote $M_0 := M(h^J)$ and $M_1 := M(h_{J+1})$.
Let $\psi \in L^2(M_0 \cup M_1)$.
By induction hypothesis there is a sequence of smooth functions 
$(f_k^0)_{k=1}^\infty$ supported in $S_{h^J}$ 
such that 
\begin{align*}
u^{f_k^0}(T) \to  1_{M_0} \psi.
\end{align*}
Moreover, there is a sequence of smooth functions 
$(f_k^1)_{k=1}^\infty$ supported in $S_{h_{J+1}}$ such that
\begin{align*}
u^{f_k^1}(T) \to  1_{M_1} (\psi - 1_{M_0} \psi).
\end{align*}
Thus
\begin{align*}
u^{f_k^0+f_k^1}(T) &\to  1_{M_0} (1 - 1_{M_1}) \psi + 1_{M_1} \psi 
= (1_{M_0 \setminus M_1} + 1_{M_1}) \psi 
\\&= 1_{M_0 \cup M_1} \psi = \psi.
\end{align*}
Moreover, $f_k^0+f_k^1$ is supported in 
$S_{h^{J}} \cup S_{h_{J+1}} \subset S_{h^{J+1}}$.
\end{proof}

\begin{proof}[Proof of Lemma \ref{lem_uniq_cont}]
Let $\psi \in L^2(M(\tau))$ and $\epsilon > 0$. 
There is a simple function 
\begin{align*}
h_\epsilon(y) = \sum_{j=1}^J T_j 1_{\Gamma_j}(y),
\end{align*}
where $J \in \N$, $T_j \in (0, T)$ and $\Gamma_j \subset \Gamma$ are open and disjoint, 
such that $\tau < h_\epsilon + \epsilon$ almost everywhere on $\Gamma$ and 
$h_\epsilon < \tau$ on $\bar \Gamma$, see e.g. \cite[proof of Lem. 4.2]{Oksanen2011}.
We denote $h_j := T_j 1_{\bar \Gamma_j}$ and define $\tau_\epsilon = h^J$ as the maximum (\ref{max_h}).
By the construction $\tau_\epsilon < \tau$
and $\tau_\epsilon \ge h_\epsilon$.

The functions (\ref{the_dense_set}) for $\tau = h_j$ are dense in $L^2(M(h_j))$ by \cite[proof of Th. 3.10]{Katchalov2001}.
Lemma \ref{lem_uniq_cont_induction} implies
that there is a smooth function 
$f$ supported in $S_{\tau_\epsilon} \subset S_\tau$ such that 
\begin{align*}
\norm{1_{M(\tau_\epsilon)} \psi - u^f(T)}_{L^2(M)}^2 < \epsilon.
\end{align*}
Thus
\begin{align}
\label{convergence_in_uniq_cont_lem}
\norm{ \psi - u^f(T)}_{L^2(M)}^2
&< \epsilon + \int_{M(\tau) \setminus M(\tau_\epsilon)} \psi^2 dV.
\end{align}
We have $V(M(\tau_\epsilon)) \to V(M(\tau))$ as $\epsilon \to 0$, see \cite{Oksanen2011}.
Thus the second term in (\ref{convergence_in_uniq_cont_lem}) tends to zero as $\epsilon \to 0$.
\end{proof}

\bigskip
{\em Acknowledgements.}
The research was supported by Finnish Centre of Excellence in Inverse
Problems Research, Academy of Finland project COE 250215,
and by European Research Council advanced grant 400803.

\bibliographystyle{abbrv} 
\bibliography{main}
\end{document}